\documentclass[12pt]{article}
\usepackage[left=1in,right=1in, top=1.2in,bottom=1.2in]{geometry}
\usepackage{here}
\usepackage{times}
\usepackage{comment}
\usepackage{amssymb,amsthm,latexsym,amsmath,epsfig,pgf}

\newtheorem{theorem}{Theorem}[section]
\newtheorem{lemma}{Lemma}[section]

\theoremstyle{remark}
\newtheorem{definition}{Definition}[section]

\theoremstyle{remark}

\begin{document}

\begin{center}
\Large{ Josephus Nim }\\
\vspace{0.5cm}
\large{Shoei Takahashi,Hikaru Manabe, Aoi Murakami and Ryohei Miyadera}\\
\end{center}
\begin{abstract}
Here, we present a variant of Nim with two piles. In the first pile, we have stones with a weight of $1$, and in the second pile, we have stones with a weight of $-2$.
Two Players take turns to take stones from one of the piles, and the total weight of stones to be removed should be equal to or less than half of the total weight of the stones in the pile.
The player who removed the last stone or stones is the winner of the game.
The authors discovered that when $(n,m)$ is a previous player's winning position, $2m+1$ is the last remaining number of the Josephus problem, where there are $n$ numbers, and every second number is to be removed. There are similar relations between the position of which the Grundy number is $s$ and the $n-s$-th removed number.
\end{abstract}

\section{Introduction}
Let $\mathbb{Z}_{\ge 0}$ and  $\mathbb{N}$ represent the sets of non-negative integers and natural numbers, respectively.

The classic game of Nim is played with stone piles. A player can remove any number of stones from any one pile during their turn; the player who takes the last stone is considered the winner. 
 
 There are many variants of the classical game of Nim. In Maximum Nim, we place an upper bound $f(n)$ on the number of stones that can be removed in terms of the number $n$ of stones in the pile (see \cite{levinenim}). The authors published their research on Maximum Nim in Miyadera et.al \cite{thaij2023b} and 
 Miyadera and Manabe \cite{integer2023}.

 In this study, we investigate a variant of Maximum Nim, with stones of a weight $1$ and a weight $-2$. 
 There are simple relations between Grundy numbers in this game and the Josephus problem. This fact is remarkable since the games of Nim and Josephus' problems are entirely different.
 This game was proposed by S. Takahashi, who is one of the authors of the present article.

\begin{definition}\label{gameoftakahashinew}
Suppose there are two piles of stones, and two players take turns removing stones from one pile. In the first pile, we have stones with a weight of $1$, and in the second pile, we have stones with a weight of $-2$.
When the total weight of stones is $m \in \mathbb{Z}_{\ge 0}$,
a player is allowed to remove stones whose total weight is less than or equal to 
$\lfloor \frac{m}{2}\rfloor$. The player who removed the last stone or stones is the winner of the game.
\end{definition}

\begin{definition}\label{takahashinew}
We denote a position of the game of Definition \ref{gameoftakahashinew} by 
$(x,y)$, where $x$ and $y$ are numbers of Type 1 stones and Type 2 stones, respectively.
\end{definition}

\begin{definition}\label{defofmexgrundy}
$(i)$ For any position $(x,y)$ of this game, there is a set of positions that can be reached by precisely one move, which we denote as \textit{move}$(x,y)$. \\	
$(x,y) = \{(x-t,y):t \leq \lfloor \frac{x-2y}{2}\rfloor\} \cup \{(x,y-t):-2u \leq \lfloor \frac{x-2y}{2}\rfloor\}$.\\
$(ii)$ The \textit{minimum excluded value} $(\textit{mex})$ of a set $S$ of non-negative integers is the smallest non-negative integer that is not in S. \\
$(iii)$ Let $\mathbf{p}$ be the position of an impartial game. The associated \textit{Grundy number} is denoted as $\mathcal{G}(\mathbf{p})$ and is recursively defined as follows:
$\mathcal{G}(\mathbf{p}) = \textit{mex}(\{\mathcal{G}(\mathbf{h}): \mathbf{h} \in move(\mathbf{p})\}).$
\end{definition}

\begin{definition}\label{NPpositions}
$(a)$ A position is referred to as a $\mathcal{P}$-\textit{position} if it is a winning position for the previous player (the player who just moved), as long as he/she plays correctly at every stage.\\
$(b)$ A position is referred to as an $\mathcal{N}$-\textit{position} if it is a winning position for the next player, as long as he/she plays correctly at every stage.
\end{definition}

\begin{definition}\label{grundysposition}
For $s,n \in \mathbb{Z}_{\ge 0}$, let
\begin{align}
\mathcal{G}_{s,n} & = \{((2s+1)\times 2^n-1+m,m):m \in \mathbb{Z}_{\ge 0} \nonumber \\
& \ \ \  \text{ such that } 0 \leq m \leq (2s+1)\times 2^n-1\},  \nonumber
\end{align}
$\mathcal{G}_{0,a}=\mathcal{G}_{0,b}=\emptyset$, and 
for $s \in \mathbb{N}$, let
\begin{equation}
\mathcal{G}_{s,a}=\{(2k,j):k \in \mathbb{Z}_{\ge 0}, 0 \leq k \leq s-1 \text{ and } 2^{s-k-1}+k \leq j \leq 2^{s-k}+k-1\}. \nonumber   
\end{equation}
\begin{equation}
\mathcal{G}_{s,b}=\{(2k+1,j):k \in \mathbb{Z}_{\ge 0}, 0 \leq k \leq s-1 \text{ and } 2^{s-k-1}+k \leq j \leq 2^{s-k}+k-1\}, \nonumber
\end{equation}
and
\begin{equation}
\mathcal{G}_s  = (\bigcup_{n=1}^{\infty} \mathcal{G}_{s,n}) \cup \mathcal{G}_{s,a} \cup \mathcal{G}_{s,b}.\nonumber
\end{equation}
\end{definition}

\begin{lemma}\label{lemmaforthree}
$(i)$ For $s,n,h \in \mathbb{Z}_{\ge 0}$ such that $h \leq s-2$ , we have the following $(\ref{h2jinc})$ and $(\ref{h21jinc})$.
\begin{equation}
\{(2h,j):h+1 \leq j \leq 2^{s-h-1}+h-1\} \subset \cup_{i=h+1}^{s-1}\mathcal{G}_{i,a} \label{h2jinc}
\end{equation}
and 
\begin{equation}
\{(2h+1,j):h+1 \leq j \leq 2^{s-h-1}+h-1\} \subset \cup_{i=h+1}^{s-1}\mathcal{G}_{i,b}. \label{h21jinc}
\end{equation}
$(ii)$ We have the following $(\ref{h2xh})$ and $(\ref{h21xh})$.
\begin{equation}
\{(2h,x):0 \leq x \leq h\} \subset \cup_{i=0}^{h} \cup _{n=0}^{\infty} \mathcal{G}_{i,n} \label{h2xh}
\end{equation}
and 
\begin{equation}
\{(2h+1,x):0 \leq x \leq h\} \subset \cup_{i=0}^{h} \cup _{n=0}^{\infty} \mathcal{G}_{i,n}. \label{h21xh}
\end{equation}
$(iii)$ For any $(x,y) \in \mathcal{G}_{s,a} \cup \mathcal{G}_{s,b}$,
$x \leq 2s-1$ and $y > \frac{x}{2}$.
\end{lemma}

\begin{proof}
$\mathrm{(i)}$ First, we prove (\ref{h2jinc}). Let $s,n,h \in \mathbb{Z}_{\ge 0}$ such that $h \leq s-2$.
\begin{align}
& \{(2h,j):h+1 \leq j \leq 2^{s-h-1}+h-1\} \nonumber \\
& \ \ = \cup_{i=h+1}^{s-1} \{(2h,j):2^{i-h-1}+h \leq j \leq 2^{i-h}+h-1\}    \nonumber \\
& \ \ = \cup_{i=h+1}^{s-1} \{(2h,j):h \leq i-1, 2^{i-h-1}+h \leq j \leq 2^{i-h}+h-1\}    \nonumber \\
& \ \ \subset \cup_{i=h+1}^{s-1}\mathcal{G}_{i,a}.\nonumber
\end{align}
Similarly, we have (\ref{h21jinc}). \\
$\mathrm{(ii)}$
For $2h+1-x \in \mathbb{N}$, there exist 
$n,t \in \mathbb{Z}_{\ge 0}$ such that $t \leq h$ and 
$2h+1-x=(2t+1)2^n$. Then,
$(2h,x)=((2t+1)2^n-1+x,x) \in \mathcal{G}_{t,n}$. Therefore, we have (\ref{h2xh}). Similarly, we prove (\ref{h21xh}).\\
$\mathrm{(iii)}$ For  $(x,y) = (2k,j)$ or $(2k+1,j)$ such that $2^{s-k-1}+k \leq j$, we have 
$y = j \geq k+1 > \frac{x}{2}$. Since $k \leq s-1$, $x \leq 2s-1$.
\end{proof}

\begin{lemma}\label{frompnotp}
Suppose that we start with a position 
\begin{equation}
(x,y) \in \mathcal{G}_s. \nonumber
\end{equation}
Then,
\begin{equation}
move(x,y) \cap \mathcal{G}_s = \emptyset. \label{nointer}
\end{equation}
\end{lemma}
\begin{proof} Suppose that we start with a position  $(x,y) \in \mathcal{G}_s$.\\
$\mathrm{(i)}$    
Suppose that $(x,y) = ((2s+1) \times 2^n-1+m,m) \in \mathcal{G}_{s,n}$, where
\begin{equation}
m \leq (2s+1) \times 2^n-1. \label{mleq2s1}    
\end{equation}
$\mathrm{(i.1)}$  Suppose that $n=0$. Then, by (\ref{mleq2s1}), 
\begin{equation}
m \leq 2s \label{mleq2s}    
\end{equation}
and $(x,y)=(2s+m,m)$, and by (\ref{mleq2s1}),
the total weight of stones is $2s+m-2m=2s-m \geq 0$.
By Definition \ref{gameoftakahashinew}, 
the total weight of the stones that can be removed is 
\begin{equation}
0 \leq \lfloor \frac{2s-m}{2} \rfloor = s - \lceil \frac{m}{2} \rceil.\label{rangeofi2b}
\end{equation}
We prove (\ref{nointer}) by contraction.\\
$\mathrm{(i.1.1)}$ 
We assume that we remove stones from the first pile, and move to the position
\begin{equation}
(u,m) =(2s+m-i,m)= ((2s+1)\times 2^k-1+m,m) \in \mathcal{G}_{s,k},\label{umisgsk}
\end{equation}
where $k \in \mathbb{Z}_{\ge 0}$ and 
\begin{equation}
0 \leq i \leq  s - \lceil \frac{m}{2}\rceil. \nonumber
\end{equation}
Then, $i = 2s-(2s+1)\times 2^k +1 \leq 0$ that 
 contradicts (\ref{umisgsk}).\\
$\mathrm{(i.1.2)}$ 
We assume that we remove stones from the second pile, and move to the position 
\begin{equation}
(u,m) =(2s+m,m-i)= ((2s+1)\times 2^k-1+m-i,m-i) \in \mathcal{G}_{s,k},\label{umisgsk22}
\end{equation}
where
\begin{equation}
i \leq m. \label{therangei22}
\end{equation}
By (\ref{umisgsk22}) and $k \geq 1$, 
$i \geq 2(2s+1)-2s-1 = 2s+1$ that contradicts (\ref{therangei22}) and (\ref{mleq2s}).\\
$\mathrm{(i.1.3)}$ 
We assume that we remove stones from the first pile, and move to 
\begin{equation}
(u,m) =(2s+m-i,m) \in \mathcal{G}_{a}\cup \mathcal{G}_{a}.\nonumber
\end{equation}
By (\ref{mleq2s}) and (\ref{rangeofi2b}),
\begin{equation}
u = 2s+m-i \geq 2s+m-(s - \lceil \frac{m}{2}\rceil) = s+m+\lceil \frac{m}{2}\rceil \geq 2m.\nonumber
\end{equation}
This contradicts $(iii)$ of Lemma \ref{lemmaforthree}.\\
$\mathrm{(i.1.4)}$ 
We assume that we remove stones from the second pile, and move to 
\begin{equation}
(u,m) =(2s+m,m-i) \in \mathcal{G}_{s,a}\cup \mathcal{G}_{s,a}.\nonumber
\end{equation}
Then, by (\ref{mleq2s}), 
\begin{equation}
u = 2s+m \geq  2m > 2(m-i).\nonumber
\end{equation}
This contradicts $(iii)$ of Lemma \ref{lemmaforthree}.\\
$\mathrm{(i.2)}$  Suppose that $n \geq 1$.
Then, the total weight of the stones is
\begin{equation}
(2s+1) \times2^n-1+m-2m=(2s+1) \times2^n-1-m, \nonumber
\end{equation}
and hence by Definition \ref{gameoftakahashinew}, the total weight of the stones that can be removed is 
\begin{equation}
0 \leq \lfloor \frac{(2s+1) \times 2^n-1-m}{2} \rfloor \leq (2s+1) \times 2^{n-1}-1.\label{howmuchremove}
\end{equation}
$\mathrm{(i.2.1)}$ Suppose we remove stones from the first pile, i.e., reduce the first coordinate $x$.
Then, by (\ref{howmuchremove}), we move to the position 
\begin{equation}
(u,m)=((2s+1) \times 2^n-1+m-i,m)\label{movetoapositon}
\end{equation}
for $i \in \mathbb{N}$ such that 
\begin{equation}
i \leq (2s+1) \times 2^{n-1}-1. \label{rangeofi}
\end{equation}
We prove (\ref{nointer}) by contradiction.\\
$\mathrm{(i.2.1.1)}$ We assume that 
\begin{equation}
((2s+1) \times 2^n-1+m-i,m) \in \cup_{k=0}^{\infty} \mathcal{G}_{s,k}.\nonumber
\end{equation}
Then, there exists $n^{\prime} \in \mathbb{Z}_{\ge 0}$ such that 
$n^{\prime} < n$ and 
\begin{equation}
((2s+1) \times 2^n-1+m-i,m) = ((2s+1) \times 2^{n^{\prime}}-1+m,m).\nonumber
\end{equation}
Then,
\begin{equation}
i \geq (2s+1)\times (2^n-2^{n^{\prime}} \geq (2s+1)\times 2^{n-1}.\nonumber
\end{equation}
that contradicts  (\ref{rangeofi}). Therefore, we have (\ref{nointer}).\\
$\mathrm{(i.2.1.2)}$ We assume that 
\begin{equation}
((2s+1) \times 2^n-1+m-i,m) \in \mathcal{G}_{s,a} \cup  \mathcal{G}_{s,b},\label{2s12n}
\end{equation}
where
\begin{equation}
i \leq (2s+1) \times 2^{n-1}-1. \label{2s12n2}
\end{equation}
Then, by (\ref{movetoapositon}), (\ref{2s12n}) and (\ref{2s12n2}),
\begin{equation}
u =(2s+1) \times 2^n-1+m-i \geq (2s+1) \times 2^{n-1}+m \nonumber
\end{equation}
Since $n \geq 1$, $(2s+1) \times 2^n-1+m-i \geq 2s$. Therefore,
this contradicts $(iii)$ of Lemma \ref{lemmaforthree}.\\
$\mathrm{(i.2.2)}$ Suppose we remove stones from the second pile, i.e., reduce the second coordinate $y$.
Since each stone in the second pile has a weight of $-2$, 
by (\ref{howmuchremove}), we can remove $i$ stones with
\begin{equation}
i \leq m \leq(2s+1) \times 2^n-1. \label{rangeofi2}
\end{equation}
Suppose that we move to the position
\begin{equation}
((2s+1)\times 2^{n}-1+m, m-i) = ((2s+1) \times 2^{n^{\prime}}-1+(m-i),m-i). \nonumber
\end{equation}
Then, $n^{\prime} \geq n+1$ and 
\begin{equation}
i = (2s+1)(2^{n^{\prime}}-2^n) \geq (2s+1)2^n,   \nonumber
\end{equation}
and this contradicts  (\ref{rangeofi2}). Therefore, we have (\ref{nointer}).\\
$\mathrm{(ii)}$ Suppose that $(x,y) = (2k,j)$ or $(2k+1,j) \in \mathcal{G}_{s,a}\cup \mathcal{G}_{s,a}$ such that 
$0 \leq k \leq s-1 $ and 
\begin{equation}
2^{s-k-1}+k \leq j \leq 2^{s-k}+k-1.\label{rangeofiwithk}
\end{equation}
$\mathrm{(ii.1)}$ We prove that 
\begin{equation}
move(x,y) \cap \mathcal{G}_{s,n} = \emptyset.\label{movesyempty2}  
\end{equation}
for any $n \in \mathbb{Z}_{\ge 0}$. By Lemma \ref{lemmaforthree},  $x \leq 2s-1$ and $u \geq 2s$ for any $(u,v) \in \mathcal{G}_{s,n}$, we have (\ref{movesyempty2}).\\
$\mathrm{(ii.2)}$.
The total weight of stones is 
$ 2k-2j$ or $ 2k+1-2j$,
and hence the total weight of stones that can be removed is
\begin{equation}
\lfloor \frac{2k-2j}{2} \rfloor =\lfloor \frac{2k+1-2j}{2} \rfloor=k-j. \label{totalremove}  
\end{equation}
By (\ref{rangeofiwithk}), 
\begin{equation}
2k-2j < 2k+1-2j \leq 1-2^{s-k} \leq 0,\nonumber
\end{equation}
we cannot remove stones from the first pile.

Next, we remove stones from the second pile, and move to the position
$(2k,i)$ or $(2k+1,i)$. We prove that $(2k,i), (2k+1,i) \notin  \mathcal{G}_{s,a} \cup \mathcal{G}_{s,b}$ by contradiction.
We suppose that 
\begin{equation}
2^{s-k-1}+k \leq i < j \leq 2^{s-k}+k-1.\label{rangeofij}
\end{equation}
Then, we remove $j-i$ stones from the second pile, and the total weight of 
stones that are removed is
\begin{equation}
-2(j-i).\label{totalji}
\end{equation}
By (\ref{totalremove}) and (\ref{totalji}) we have
\begin{equation}
k-j \geq -2j+2i, \nonumber
\end{equation}
and hence by (\ref{rangeofij})
\begin{equation}
k+j \geq 2i \geq 2^{s-k}+2k.  \nonumber
\end{equation}
Then we have 
\begin{equation}
j \geq 2^{s-k}+k \ \nonumber
\end{equation}
that contradicts (\ref{rangeofij}).
Therefore, we have (\ref{nointer}).
\end{proof}

\begin{lemma}\label{fromnotptop1}
We suppose that we start with a position 
\begin{equation}
(x,y) \notin \mathcal{G}_s \nonumber
\end{equation}
such that  $x \geq 2s$.
Then,
\begin{equation}
move(x,y) \cap \mathcal{G}_s \ne \emptyset. \nonumber
\end{equation}
\end{lemma}

\begin{proof}
Suppose that we start with a position 
\begin{equation}
(x,y) \notin \mathcal{G}_s \nonumber
\end{equation}
such that  $x \geq 2s$.
Then,
there exists $n \in \mathbb{Z}_{\ge 0}$ such that 
\begin{equation}
(2s+1)\times 2^n \leq x+1 <(2s+1)\times 2^{n+1}. \label{rangeofx}
\end{equation}
Let 
\begin{equation}
m = x-((2s+1)\times 2^n-1).  \label{meqx}  
\end{equation}
 Then, by (\ref{rangeofx})
\begin{align}
0 & \leq m \nonumber \\
& =  x-((2s+1)\times 2^{n-1}-1)   \nonumber \\
& \leq (2s+1)\times 2^{n+1} -2 -((2s+1)\times 2^n-1)  \nonumber  \\
& =(2s+1)\times 2^n-1.  \label{mranze}
\end{align}
Therefore, by (\ref{meqx}) and (\ref{mranze})
\begin{equation}
x \geq 2m.  \label{xbigm}
\end{equation}
The total weight of stones is $x-2y$, and we can remove stones whose total weight is
\begin{equation}
\frac{x-2y}{2}. \label{removeweight2}   
\end{equation}
If $y=m$,
$(x,y)=((2s+1) \times 2^n-1+m,m) \in \mathcal{G}_s$. 
Therefore, we have two cases $(i)$ and $(ii)$.\\
$\mathrm{(i)}$ Suppose that 
\begin{equation}
y > m.  \label{ygreatm}  
\end{equation}
By (\ref{xbigm}) and  (\ref{ygreatm}) we have
\begin{equation}
x+2y > 4m,\nonumber
\end{equation}
and hence
\begin{equation}
\frac{x-2y}{2} > -2(y-m) \nonumber
\end{equation}
Therefore, by (\ref{removeweight2}), we can remove $y-m$ stones from the second pile to move to
$((2s+1)\times 2^n-1+m,m)  \in \mathcal{S}$.\\
$\mathrm{(ii)}$ Suppose that 
\begin{equation}
y < m.  \nonumber
\end{equation}
By (\ref{xbigm}), 
\begin{align}
 & \frac{x-2y}{2}-(m-y) \nonumber \\
& = \frac{x-2y-2(m-y)}{2}  \nonumber \\
& = \frac{x-2m}{2}  \nonumber \\
& \geq 0.\label{2ngreatmy2}
\end{align}
By (\ref{2ngreatmy2}), 
\begin{equation}
\frac{x-2y}{2} \geq m-y,  \nonumber
\end{equation}
and hence 
we can remove $m-y$ stones from the first pile to move to
the position $(x,m)=((2s+1)\times 2^n-1+m,m) \in  \mathcal{G}_s$.\\
\end{proof}

\begin{lemma}\label{fromnotptop2}
$(i)$ If 
\begin{equation}
(x,y) \notin \cup_{i=0}^s \mathcal{G}_i   \nonumber
\end{equation}
and $x \leq 2s-1$,
then 
\begin{equation}
move(x,y) \cap \mathcal{G}_s \ne \emptyset.   \nonumber
\end{equation}
\end{lemma}

\begin{proof}
Let 
\begin{equation}
(x,y) \notin \cup_{i=0}^s \mathcal{G}_i \label{notingi}
\end{equation}
and $x \leq 2s-1$.
Then there exists $k$ such that
$0 \leq k \leq s-1$ and $x = 2k$ or $x=2k+1$.\\
$\mathrm{(i)}$ First we suppose that $x=2s-1$ or $2s-2$.
By Definition \ref{grundysposition}
\begin{equation}
(2s-1,s), (2s-2,s) \in \mathcal{G}_{s,a} \cup \mathcal{G}_{s,b}.\label{2s12s1}
\end{equation}
By $(ii)$ of Lemma \ref{lemmaforthree},
\begin{equation}
\{(2s-1,j):0 \leq j \leq s-1\} \subset \cup_{i=0}^{s-1} \cup _{n=0}^{\infty} \mathcal{G}_{i,n} \label{2sminas1}
\end{equation}
and 
\begin{equation}
\{(2s-2,j):0 \leq j \leq s-1\} \subset \cup_{i=0}^{s-1} \cup _{n=0}^{\infty} \mathcal{G}_{i,n}. \label{2sminas2}
\end{equation}
By (\ref{notingi}), (\ref{2s12s1}), (\ref{2sminas1}) and (\ref{2sminas2}), we assume that 
\begin{equation}
y \geq s+1.\label{ygeqs1}
\end{equation}
At the position $(2s-1,y)$ or $(2s-2,y)$, the total weight of stones is $2s-1-2y$ or $2s-2-2y$, and
we can remove the total weight of $s-y-1$ stones.
By (\ref{ygeqs1}), 
\begin{equation}
s-y-1 \geq -2y+2s = -2(y-s).
\end{equation}
Therefore, we remove $y-s$ stones from the second pile to
move to $(2s-1,s)$ or $(2s-2,s) \in \mathcal{G}_{s,a} \cup \mathcal{G}_{s,b}$.\\
$\mathrm{(ii)}$ Next, we suppose that $x=2k$ or $x=2k+1$ with $k \leq s-2$.
By $(i)$ and $(ii)$ of Lemma \ref{lemmaforthree}, 
\begin{equation}
\{(2k,j):0 \leq j \leq 2^{s-k-1}+k-1\} \subset \cup_{i=0}^{s-1} \mathcal{G}_i \label{from0tos1a}
\end{equation}
and 
\begin{equation}
\{(2k+1,j):0 \leq j \leq 2^{s-k-1}+k-1\} \subset \cup_{i=0}^{s-1} \mathcal{G}_i.\label{from0tos1b}
\end{equation}
By Definition \ref{grundysposition},
\begin{equation}
\{(2k,j):2^{s-k-1} +k\leq j \leq 2^{s-k}+k-1\} \subset  \mathcal{G}_s.\label{from0tos1c}
\end{equation}
and 
\begin{equation}
\{(2k+1,j):2^{s-k-1} +k\leq j \leq 2^{s-k}+k-1\} \subset \mathcal{G}_s.\label{from0tos1d}
\end{equation}
By (\ref{notingi}), (\ref{from0tos1a}),  (\ref{from0tos1b}),  (\ref{from0tos1c}), and (\ref{from0tos1d}),
\begin{equation}
y \geq 2^{s-k}+k. \label{ygeq2sk}
\end{equation}
At the position $(x,y)$, the total weight of stones is
$x-2y=2k-2y$ or $x-2y=2k+1-2y$, and 
and the total weight of stones that can be removed is
\begin{equation}
\lfloor \frac{2k+1-2y}{2} \rfloor = \lfloor \frac{2k-2y}{2} \rfloor = k-y.\label{totalremove2}
\end{equation}
We prove that we can move to the position $(2k,2^{s-k-1}+k)$ or $(2k+1,2^{s-k-1}+k)$.
By (\ref{ygeq2sk}),
\begin{align}
k-y -2(2^{s-k-1}+k-y) & = k-y-2^{s-k}-2k+2y  \nonumber \\
& = y-k-2^{s-k} \geq 0,
 \nonumber
\end{align}	
and hence, we can move to $(2k,2^{s-k-1}+k)$ or 
$(2k+1,2^{s-k-1}+k) \in \mathcal{G}_s$ by removing $y-(2^{s-k-1}+k)$ stones from the second pile.
\end{proof}

\begin{theorem}\label{theforgsset}
$\mathcal{G}_s$ is the set of positions whose Grundy number is $s$.
\end{theorem}
\begin{proof}
This is direct from Definition \ref{defofmexgrundy}, Lemma \ref{frompnotp}, Lemma \ref{fromnotptop1}, and Lemma \ref{fromnotptop2}.
\end{proof}

\section{Josephus Problem}

\begin{definition}
We have a finite sequence $1,2,3,4, \cdots, v$ arranged in a circle, and we start with $2$ to remove every second number until there is only one number left.
This is a well-known Josephus problem, and we denote the number removed in this order by $e_1=2, e_2=4, \cdots, e_{v-1}$, and we denote the number left by $e_v$.
For any $v$, $F_s(v)=e_{v-s}$ for $s=0,1,2, \cdots, v-1$. Note that $F_0(v)=e_v$ is the last number remains in the process of elimination. 
\end{definition}

\begin{lemma}\label{fsvalues}
We have the following formulas.
\begin{equation}
F_s(s+k)  = 2k  
\end{equation}
for any $1 \leq k \leq s$, and 
\begin{equation}
F_s(2s+1)  = 1. 
\end{equation}
\end{lemma}
\begin{proof}
Since $1 \leq k \leq s$, we have numbers $1,2, \cdots, 2k, \cdots ,s+k$. We remove numbers 
\begin{equation}
2,4, \cdots 2k, \cdots . \label{howtoremove1}   
\end{equation}
We denote the number removed in this order by $e_1=2, e_2=4, \cdots, e_{s+k-1}$, and the $e_{s+k}$ is the last number that remains. $F_s(s+k)=e_{s+k-s}=e_k$ that is the $k$-th number to be removed. Therefore, by (\ref{howtoremove1}),
we have $F_s(s+k)=e_{s+k-s}=e_k=2k$.

When we have $1,2, \cdots, 2s+1$, we remove numbers $2,4, \cdots,2s$ in the first time around the circle, and the
$s$-th removed number is $e_s = 2s$. In the second time around the circle, $1$ will be removed and $e_{s+1}=1$.
$F_s(2s+1)=e_{2s+1-s}=e_{s+1}=1$.
\end{proof}

\begin{lemma}\label{joserecur}
We have the following recursions.
\begin{equation}
F_s(2v)  = 2F_s(v)-1 \label{fs2v} 
\end{equation}
and
\begin{equation}
F_s(2v+1)  = 2F_s(v)+1. \label{fs2v2}     
\end{equation}
\end{lemma}
\begin{proof}
First, we prove these recursions for $F_0$.
We assume that 
\begin{equation}
F_0(v)=x. \label{f0v}
\end{equation}
Then, $x$ is the last number left when we start with numbers $1,2, \cdots, v$.
Suppose that we start with numbers $1,2,3, \cdots, 2v$. When all even numbers are removed for the first time around the circle, $v$ 
 numbers $1,3, \cdots, 2v-1$ remains until the last. By (\ref{f0v}), the $x$-th number among these $1,3, \cdots, 2v-1$ will be the last remaining number in this Josephus problem, and this $x$-th number is $2x-1$. Therefore we have
\begin{equation}
F_0(2v)  = 2F_0(v)-1.   
\end{equation}
Suppose that we start with numbers $1,2,3, \cdots, 2v+1$. when all even numbers are removed in the first time around the circle, and the number $1$ is removed at the beginning of the second time around the circle, $v$ numbers $3,5, \cdots, 2n+1$ remain. By (\ref{f0v}), the $x$-th number among these $3,5, \cdots, 2n+1$ will survie, and this $x$-th number is  $2x+1$. Therefore we have
\begin{equation}
F_0(2v+1)  = 2F_0(v)+1.   
\end{equation}
By a method that is similar to the one used for $F_0$, we prove
(\ref{fs2v}) and (\ref{fs2v2}).
\end{proof}

\begin{theorem}\label{theoremforjosephus}
If $v=(2s+1)2^n+m$ such that $s,n,m \in \mathbb{Z}_{\ge 0}$ and $0 \leq m \leq (2s+1)2^n-1$, then 
\begin{equation}
F_s(v)=2m+1\label{fsv2m1}
\end{equation}
\end{theorem}
\begin{proof}
We prove this by mathematical induction.
By Lemma \ref{fsvalues} and Lemma \ref{joserecur},
for $m=2k$ with $k \in \mathbb{Z}_{\ge 0}$, 
\begin{align}
F_s((2s+1)+m) & =F_s(2s+2k+1)\nonumber \\
& =2F_s(s+k)+1 \nonumber \\
&=2(2k)+1=2m+1\nonumber 
\end{align}	
 and for $m=2k+1$ with $k \in \mathbb{Z}_{\ge 0}$,
\begin{align}
F_s((2s+1)+m) & =F_s(2s+2k+2)\nonumber \\
& =2F_s(s+k+1)-1\nonumber \\
& =2(2(k+1))-1\nonumber \\
& =2(2k+1)+1=2m+1. \nonumber 
\end{align}	
We assume that there exist $n_0$ and $m_0$ such that (\ref{fsv2m1}) is valid for any $n \leq n_0$ and $m \leq m_0$.
If $m_0 + 1 = 2m+1$ for $m \leq m_0$, by Lemma \ref{joserecur} and the assumption of mathematical induction, 
\begin{align}
F_s((2s+1)2^{n_0+1}+m_0+1) & = F_s((2s+1)2^{n_0+1}+2m+1)\nonumber \\
& = 2F_s((2s+1)2^{n_0}+m)+1\nonumber \\
& = 2(2m+1)+1 \nonumber \\
& = 2(m_0+1)+1. \nonumber 
\end{align}	
If $m_0 + 1 = 2m$ for $m \leq m_0$, by Lemma \ref{joserecur} and the assumption of mathematical induction, 
\begin{align}
F_s((2s+1)2^{n_0+1}+m_0+1) & = F_s((2s+1)2^{n_0+1}+2m)\nonumber \\
& = 2F_s((2s+1)2^{n_0}+m)-1\nonumber \\
& = 2(2m+1)-1 \nonumber \\
& = 2(2m)+1 \nonumber \\
& = 2(m_0+1)+1. \nonumber 
\end{align}	

\end{proof}

\begin{theorem}\label{theoremforrelation}
For $(x,y) \in \mathcal{G}_{s,n}$, 
\begin{equation}
F_s(x+1)=2y+1.    
\end{equation}
\end{theorem}
\begin{proof}
By Definition \ref{grundysposition}, for $(x,y) \in \mathcal{G}_{s,n}$ there exist 
$m \in  \mathbb{Z}_{\ge 0}$ such that $0 \leq m \leq (2s+1)\times 2^n-1$ and 
$(x,y)=((2s+1)2^n-1+m,m)$. Then by Theorem \ref{theoremforjosephus}, 
\begin{equation}
F_s(x+1)=F_s((2s+1)2^n+m)=2m+1 = 2y+1. 
\end{equation}
\end{proof}
 By Theorem \ref{theforgsset}, Theorem \ref{theoremforjosephus}, and Theorem \ref{theoremforrelation},
 there are simple relations between the positions whose Grundy number is $s$ and the numbers that will be 
$s$-th number removed from the last.

\bibliographystyle{amsplain}

\begin{thebibliography}{99}

\bibitem{levinenim} L. Levine, Fractal sequences and restricted Nim, {\it Ars Combinatoria}  {\bf 80} (2006), 113--127.
\bibitem{thaij2023b}R. Miyadera, S. Kannan and H. Manabe, Maximum Nim and Chocolate Bar Games, {\it Thai Journal of Mathematics}, accepted.
\bibitem{integer2023}  R. Miyadera and H. Manabe, Restricted Nim with a Pass, {\it Integers },  {\bf Vol.23}, (2023), $\#$ G3.
\bibitem{computerscience} Graham, R. L.; Knuth, D. E.; Patashnik, O. (1989). Concrete Mathematics: A Foundation for Computer Science. Addison Wesley. ISBN 978-0-201-14236-5.
\end{thebibliography}


\end{document}